\documentclass[reqno, 10pt]{amsart}
\usepackage{color}
\usepackage{amsmath}
\usepackage{amsfonts}
\usepackage{amssymb}
\usepackage{amsthm}
\usepackage{newlfont}
\usepackage{graphicx}

\newtheorem{thm}{Theorem}[section]

\newtheorem{prop}[thm]{Proposition}
\theoremstyle{definition}

\theoremstyle{remark}
\newtheorem{rem}[thm]{Remark}

\numberwithin{equation}{section}
\numberwithin{thm}{section}

\newcommand{\ed}{\end {document}}

\newcounter{smalllist}

\title[global solution for turbulent flow equations]{Global existence of smooth
solutions to three-dimensional turbulent flow equations}

\author[D. Bian]{Dongfen Bian}
\address{The Graduate School of China Academy of Engineering Physics,
P. O. Box 2101,\ Beijing 100088,\ PR China,}%
\email{bian\_dongfen@mail.com}
\author[B. Guo]{Boling Guo}
\address{Institute of Applied Physics and Computational Mathematics,
 P. O. Box 8009,\ Beijing 100088,\ PR China.}
\email{gbl@iapcm.ac.cn}

\begin{document}
\maketitle
\begin{abstract}
In this paper we are concerned with the global existence of smooth solutions to
 the turbulent flow equations for compressible flows in $\mathbb{R}^3$.
  The global well-posedness is proved under the condition that
 the initial data are close to the standard equilibrium state in $H^3$-framework.
 The proof relies on energy estimates about velocity, temperature,
 turbulent kinetic energy and rate of viscous dissipation. We use
 several new techniques to overcome the difficulties from the
  product of two functions and higher order norms. This is the first
  result concerning $k$-$\varepsilon$ model equations.

\end{abstract}
 \vspace{.2in} {\bf  {Key words and phrases:}}{
Turbulent flow equations, compressible flows, $k$-$\varepsilon$ model,
smooth solution, global existence.}

\setlength{\baselineskip}{20pt}
\section{Introduction}

We consider in this work the turbulent flow equations for compressible flows on $\mathbb{R}^3$,
\begin{align}\label{MHD}
\begin{cases}
\rho_t+ \mbox{div}(\rho u)=0,\\
(\rho u)_t+ \mbox{div}(\rho u\otimes u)-\Delta u-\nabla\mbox{div}u+\nabla p=-\frac{2}{3}\nabla(\rho k),\\
(\rho h)_t+\mbox{div}(\rho u h)-\Delta h=\frac{Dp}{Dt}+S_k,\\
(\rho k)_t+ \mbox{div}(\rho u k)-\Delta k=G-\rho \varepsilon,\\
(\rho \varepsilon)_t+ \mbox{div}(\rho u \varepsilon)-\Delta\varepsilon
=\frac{C_1G \varepsilon}{k}-\frac{C_2 \rho \varepsilon^2}{k},\\
(\rho, u, h, k,\varepsilon)(x,t)|_{t=0}=(\rho_0(x), u_0(x), h_0(x), k_0(x), \varepsilon_0(x)),
\end{cases}
\end{align}
 with $S_k=[\mu(\frac{\partial u^i}{\partial x_j}+\frac{\partial u^j}{\partial x_i})
 -\frac{2}{3}\delta_{ij}\mu\frac{\partial u^k}{\partial x_k}]\frac{\partial u^i}{\partial x_j}
 +\frac{\mu_t}{\rho^2}\frac{\partial p}{\partial x_j}\frac{\partial \rho}{\partial x_j}$, $G
 =\frac{\partial u^i}{\partial x_j}[\mu_e(\frac{\partial u^i}{\partial x_j}
 +\frac{\partial u^j}{\partial x_i})-\frac{2}{3}\delta_{ij}(\rho k+\mu_e\frac{\partial u^k}{\partial x_k})]$,
  where $\delta_{ij}=0$ if $i\neq j$, $\delta_{ij}=1$ if $i=j$,
  and $\mu,\ \mu_t,\ C_1,\ C_2$ are four positive constants satisfying $\mu+\mu_t=\mu_e$.

 Here $\rho,\  u,\  h,\ k$ and $\varepsilon$
 denote the density, velocity, total enthalpy, turbulent kinetic energy and rate of viscous dissipation, respectively.
  The pressure $p$ is a smooth function of $\rho$. In this paper, without loss of generality,
  we have renormalized some constants to be 1.

 The system (\ref{MHD}) is formed by combining effect of turbulence on
 time-averaged Navier-Stokes equations with the $k$-$\varepsilon$ model equations.
  Turbulence stands out as a prototype of multi-scale phenomenon that occurs in nature.
   It involves wide ranges of
  spatial and temporal scales which makes it very difficult to study analytically
  and prohibitively expensive to simulate computationally.
  Up to now, there are not any general theory suitable for the turbulent flows.
    Many, if not most, flows of engineering significance are turbulent, so the
    turbulent flow regime is not just of theoretical interest.
   Hence, research about the above system (\ref{MHD}) is initial and very important.

   There are many results about the simple model, such as classical Navier-Stokes equations
   and incompressible Navier-Stokes equations with variable density.

  For classical Navier-Stokes equations (let $\rho=C$, $k=\varepsilon=0$),
    there are many authors doing research about them, but
   the uniqueness and regularity of weak solutions for general initial data are
open.  In particular, Fujita and Kato \cite{f1, f2} proved the local
well-posedness of 3D Navier-Stokes equations for large initial data and the global
 well-posedness for small initial data in the
homogeneous Sobolev space $\dot{H}^{\frac{1}{2}}$.
If taking $u_0$ small with respect to the viscosity in the Lebesgue space $L^N(\mathbb{R}^N)$,
 the classical Navier-Stokes equations are globally well-posed
(see \cite{giga1, kato}). This result has been adapted to the case of bounded domains by Y. Giga and T. Miyakawa
in \cite{giga2} and to the case of exterior domains by Y. Giga and H. Sohr in \cite{giga3}, and H. Iwashita
in \cite{iwashita}. In the half-space case, the well-posedness issue has been studied by H. Kozono in \cite{kozono}
(see also \cite{can3} for results related to critical Besov spaces with negative index of regularity).
 In the absence of heat
conduction, it was proved by Z. Xin that any non-zero smooth solution
with initial compactly supported density would blow up in finite
time(see \cite{xin}). As a reasonable starting point, we will
therefore restrict our work to solutions such that $\rho$ remains
positive.

For incompressible Navier-Stokes equations with variable density (let $k=\varepsilon=0$), there are also many results.
Specially, O. Ladyzhenskaya and V. Solonnikov \cite{o-v} studied the existence of strong (smooth) solutions with positive density
whereas the theory of global weak solutions with finite energy has been performed in the book \cite{lions1}.
  Abidi \cite{Abidi1} and Danchin-Mucha \cite{d6}-\cite{d8} have obtained several results about
incompressible Navier-Stokes
equations with variable density in the critical spaces.
 For more about well-posedness in other critical spaces, readers can
 refer to \cite{Bian, Bian2, can, can2, c1, d1, d2, d3, d4, d5, mey}.

For the system (\ref{MHD}), there is not any result. This paper is devoted to the study of the global existence of smooth solutions for the
system (\ref{MHD}) under suitable assumptions.
Our result is expressed in the following.


\begin{thm}\label{thm1}
Assume that initial data are close enough to the constant state
$(\bar{\rho},0,0,\bar{k},0)$, i.e. there exists a constant $\delta_0$ such that
if
\begin{equation} \|(\rho_0-\bar{\rho}, u_{0}, h_0, k_0-\bar{k}, \varepsilon_0)\|_{H^3(\mathbb{R}^3)}\leq \delta_0,
 \end{equation}
 then the system
(\ref{MHD}) admits a unique smooth solution $(\rho,u,h,k,\varepsilon)$ such that for any
$t\in [0,\infty)$,
\begin{equation*}
\begin{split}
&\|(\rho-\bar{\rho}, u, h, k-\bar{k}, \varepsilon)\|_{H^{3}}^2
+\int_0^t\|\nabla \rho\|_{H^2}^2+\|(\nabla u, \nabla h,\nabla k, \nabla \varepsilon)\|_{H^3}^2\mbox{d}s\\
&\leq
C\|(\rho_0-\bar{\rho},u_0, h_0, k_0-\bar{k}, \varepsilon_0)\|_{H^{3}}^2,
\end{split}
\end{equation*}
where $C$ is a positive constant.
\end{thm}

\begin{rem}
The existence of the
local solution for (\ref{MHD}) can be obtained from
the standard method based on the Banach theorem and contractivity of the operator
defined by the linearization of the problem on a small time interval (see also \cite{Nishita, J.N, volpert}).
Hence, we omit the local existence part for simplicity. The global existence
of smooth solutions will be proved by extending
the local solutions with respect to time based on a-priori global
estimates.
\end{rem}
\begin{rem}
We use several new techniques to overcome the difficulties from the product of two functions and higher order norms.
In a separate paper we consider the convergence rate of smooth solution to the constant state.
\end{rem}

 \textbf{Notation:} Throughout the paper, $C$ stands for a
general constant, and may change from line to line.
The norm $\|(A,B)\|_{X}$ is equivalent to $\|A\|_{X}+\|B\|_{X}$.
The notation $L^p(\mathbb{R}^3)$, $1\leq p\leq\infty$, stands for
 the usual Lebesgue spaces on $\mathbb{R}^3$ and $\|\cdot\|_p$ denotes its $L^p$ norm.

\section{A priori estimates}
In this section we establish a-priori estimates of
 the solutions. we assume that $(v,u,\theta,k,\varepsilon)$ is a smooth solution to
(\ref{MHD}) on the time interval $(0,T)$ with $T>0$. We shall establish the following proposition.
\begin{prop}\label{small theorem}
There exists a constant $\delta\ll1$ such that if
 \begin{equation}\label{smallcondition1}
\sup_{0\leq t\leq T}\|(\rho-\bar{\rho}, u, h, k-\bar{k}, \varepsilon)\|_{H^{3}}\leq \delta,
 \end{equation}
 then for any $t\in[0,T]$, there exists a constant $C_1>1$ such that it holds
\begin{equation}\label{smallcondition2}
\begin{split}
&\|(\rho-\bar{\rho}, u,h, k-\bar{k}, \varepsilon)\|_{H^{3}}^2
+\int_0^t\|\nabla \rho\|_{H^2}^2+\|(\nabla u,\nabla h, \nabla k, \nabla\varepsilon
)\|_{H^3}^2\mbox{d}s\\
&\leq
C_1\|(\rho_0-\bar{\rho}, u_0, h_0, k_0-\bar{k}, \varepsilon_0)\|_{H^{3}}^2.
\end{split}
 \end{equation}
\end{prop}
\begin{proof}
First, let $\rho=a+\bar{\rho}$, $k=m+\bar{k}$, $f'(\rho)=\frac{p'(\rho)}{\rho}$,
 we rewrite the
system (\ref{MHD}) as follows:
\begin{align}\label{MHD1}
  \begin{cases}
   a_t+ \mbox{div}((a+\bar{\rho})u)=0,\\
   u_t+  u\cdot \nabla u-\frac{1}{a+\bar{\rho}}(\Delta u+\nabla\mbox{div}u)
   +\nabla [f(a+\bar{\rho})-f(\bar{\rho})]\\=-\frac{2}{3(a+\bar{\rho})}\nabla((a+\bar{\rho}) (m+\bar{k})),\\
 h_t+ u \cdot \nabla h-\frac{1}{a+\bar{\rho}}\Delta h=-f'(a+\bar{\rho})(a+\bar{\rho})\mbox{div}u+\frac{1}{a+\bar{\rho}}S_k,\\
  m_t+ u\cdot\nabla m-\frac{1}{a+\bar{\rho}}\Delta m=\frac{1}{a+\bar{\rho}}G- \varepsilon,\\
    \varepsilon_t+ u \cdot\nabla\varepsilon-\frac{1}{a+\bar{\rho}}\Delta\varepsilon
    =\frac{C_1G \varepsilon}{(a+\bar{\rho})(m+\bar{k})}-\frac{C_2 \varepsilon^2}{m+\bar{k}},\\
   (a, u, h, m,\varepsilon)(x,t)|_{t=0}=(a_0(x), u_0(x), h_0(x), m_0(x), \varepsilon_0(x)),
\end{cases}
\end{align}
with $S_k=[\mu(\frac{\partial u^i}{\partial x_j}+\frac{\partial u^j}{\partial x_i})
-\frac{2}{3}\delta_{ij}\mu\frac{\partial u^k}{\partial x_k}]\frac{\partial u^i}{\partial x_j}
+\frac{\mu_t}{(a+\bar{\rho})^2}\frac{\partial p}{\partial x_j}\frac{\partial a}{\partial x_j}$,
 $G=\frac{\partial u^i}{\partial x_j}[\mu_e(\frac{\partial u^i}{\partial x_j}
 +\frac{\partial u^j}{\partial x_i})-\frac{2}{3}\delta_{ij}((a+\bar{\rho})(m+ \bar{k})+\mu_e\frac{\partial u^k}{\partial x_k})]$.

From a priori
assumption (\ref{smallcondition1}) and the Sobolev inequality together with
the first equation of (\ref{MHD1}), we have
\begin{equation}\label{smallcondition3}
\sup_{x\in \mathbb{R}}|(a, a_t, \nabla a, u, \nabla u,  h, \nabla h, m, \nabla m, \varepsilon, \nabla \varepsilon)|
\leq C \|(a,u,h,m,\varepsilon)\|_{H^3}\leq C \delta.
\end{equation}
Moreover,
\begin{equation}\label{density3}
\frac{\bar{\rho}}{2}\leq \rho=a+\bar{\rho}\leq 2\bar{\rho},\ \ \  \frac{\bar{k}}{2}\leq k=m+\bar{k}\leq 2\bar{k},
\end{equation}
and
\begin{equation}\label{density4}
0<\frac{1}{C_0}\leq f'(\rho)\leq C_0<\infty,\ \ \ |f^{(n)}(\rho)|\leq C_0\ \  \mbox{for any positive integer}\  \ n,
\end{equation}
with $C_0$ a positive constant.

In what follows, we will always use the smallness assumption of $\delta$ and (\ref{smallcondition3})-(\ref{density4}).
 We divide the a priori
estimates into three steps.

Step 1: $L^2$-norms of $u,\ h,\ m,\ \varepsilon$

Multiplying the second equation of (\ref{MHD1}) by $u$ and integrating over $\mathbb{R}^3$,
 one can deduce that
\begin{equation}\label{smallvelocity1}
\begin{split}
&\frac{1}{2}\frac{d}{dt}\|u\|_{2}^{2}+\frac{1}{a+\bar{\rho}}(\|\nabla u\|_2^2+\|\mbox{div}u\|^2_2)\\
&=-\int_{\mathbb{R}^3}u\cdot \nabla u\cdot u\mbox{d}x
+\int_{\mathbb{R}^3}(f(a+\bar{\rho})-f(\bar{\rho}))\mbox{div}u\mbox{d}x\\
&+\int_{\mathbb{R}^3}\frac{1}{(a+\bar{\rho})^2}[(\nabla a \otimes u):\nabla u+\mbox{div}u(\nabla a \cdot  u)]\mbox{d}x\\
&-\int_{\mathbb{R}^3}\frac{2}{3(a+\bar{\rho})}\nabla((a+\bar{\rho})(m+\bar{k}))\cdot u\mbox{d}x.
\end{split}
\end{equation}
 With the help of H\"{o}lder's inequality, $\dot{W}^{1,2}(\mathbb{R}^3)\hookrightarrow L^6(\mathbb{R}^3)$
  and (\ref{smallcondition3}), we estimate the right-hand side of (\ref{smallvelocity1}) as
 \begin{equation}\label{velocity1}
 -\int_{\mathbb{R}^3}u\cdot \nabla u\cdot u\mbox{d}x\leq C\|u\|_6\|\nabla u\|_2\|u\|_3\leq C\delta\|\nabla u\|_2^2,
 \end{equation}
 \begin{equation}\label{velocity2}
  \begin{split}
 &\int_{\mathbb{R}^3}\frac{1}{(a+\bar{\rho})^2}[(\nabla a \otimes u):\nabla u+\mbox{div}u(\nabla a \cdot  u)]\mbox{d}x\\
 &\leq C\|\nabla a\|_2\|\nabla u\|_2\|u\|_{L^{\infty}}
 \leq C\delta\|(\nabla a,\nabla u)\|_2^2,
  \end{split}
 \end{equation}
 \begin{equation}\label{velocity3}
 \begin{split}
 &-\int_{\mathbb{R}^3}\frac{2}{3(a+\bar{\rho})}\nabla((a+\bar{\rho})(m+\bar{k}))\cdot u\mbox{d}x\\
 &=-\int_{\mathbb{R}^3}\frac{2(m+\bar{k})}{3(a+\bar{\rho})}\nabla a\cdot u
 -\frac{2}{3(a+\bar{\rho})}(a+\bar{\rho})\nabla m\cdot u\mbox{d}x
 \\&\leq C\delta\|(\nabla a,\nabla m)\|_2^2,
 \end{split}
 \end{equation}
 \begin{equation}\label{velocity4}
 \begin{split}
& \int_{\mathbb{R}^3}(f(a+\bar{\rho})-f(\bar{\rho}))\mbox{div}u\mbox{d}x\\
&=-\int_{\mathbb{R}^3}[f(a+\bar{\rho})-f(\bar{\rho})](\frac{a_t+u\cdot\nabla a}{a+\bar{\rho}})\mbox{d}x\\
 &=-\frac{\mbox{d}}{\mbox{d}t}\int_{\mathbb{R}^3}F(a)\mbox{d}x-\int_{\mathbb{R}^3}\frac{f'(\bar{\rho}+\theta a)}{a+\bar{\rho}}au\cdot\nabla a\mbox{d}x\\
 &\leq -\frac{\mbox{d}}{\mbox{d}t}\int_{\mathbb{R}^3}F(a)\mbox{d}x+C\delta\|\nabla a\|_2^2,
 \end{split}
 \end{equation}
 with $\theta \in (0,1)$ and $F(a)$ is defined as
 \begin{equation}\label{fensityfunction}
 F(a)=\int_0^a\frac{f(s+\bar{\rho})-f(\bar{\rho})}{s+\bar{\rho}}\mbox{d}s.
 \end{equation}
 Combining (\ref{velocity1})-(\ref{velocity4}) with (\ref{smallvelocity1}), we get
 \begin{equation}\label{velocity}
\frac{1}{2}\frac{d}{dt}[\|u\|_{2}^{2}+\int_{\mathbb{R}^3}F(a)\mbox{d}x]
+\frac{1}{a+\bar{\rho}}(\|\nabla u\|_2^2+\|\mbox{div}u\|^2_2)\leq C\delta \|(\nabla a, \nabla u, \nabla m)\|_2^2.
\end{equation}

Multiplying the energy equation, governing equation for turbulent kinetic energy $k$ and $\varepsilon$-equation of (\ref{MHD1}) by $h$, $m$ and $\varepsilon$ respectively, and integrating them over the whole space $\mathbb{R}^3$, similarly we
can get that
\begin{equation}\label{smalltemperature1}
\begin{split}
&\frac{1}{2}\frac{d}{dt}\|h\|_{2}^{2}+\frac{1}{a+\bar{\rho}}\|\nabla h\|_2^2\\
&=\int_{\mathbb{R}^3}\frac{1}{(a+\bar{\rho})^2}\nabla a\cdot\nabla h\cdot h\mbox{d}x
-\int_{\mathbb{R}^3}f'(a+\bar{\rho})(a+\bar{\rho})\mbox{div}u h\mbox{d}x
\\&+\int_{\mathbb{R}^3}\frac{1}{a+\bar{\rho}}S_k\cdot h\mbox{d}x-\int_{\mathbb{R}^3}u\cdot\nabla h\cdot h\mbox{d}x,
\end{split}
\end{equation}
\begin{equation}\label{smallkinetic1}
\begin{split}
&\frac{1}{2}\frac{d}{dt}\|m\|_{2}^{2}+\frac{1}{a+\bar{\rho}}\|\nabla m\|_2^2\\
&=\int_{\mathbb{R}^3}\frac{1}{(a+\bar{\rho})^2}\nabla m\cdot\nabla a\cdot m\mbox{d}x
-\int_{\mathbb{R}^3}\varepsilon\cdot m\mbox{d}x
+\int_{\mathbb{R}^3}\frac{1}{a+\bar{\rho}}G\cdot m\mbox{d}x\\&-\int_{\mathbb{R}^3}u\cdot\nabla m\cdot m\mbox{d}x,
\end{split}
\end{equation}
\begin{equation}\label{smallvarepsilon1}
\begin{split}
&\frac{1}{2}\frac{d}{dt}\|\varepsilon\|_{2}^{2}+\frac{1}{a+\bar{\rho}}\|\nabla \varepsilon\|_2^2\\
&=\int_{\mathbb{R}^3}\frac{1}{(a+\bar{\rho})^2}\nabla \varepsilon\cdot\nabla a\cdot \varepsilon\mbox{d}x
+\int_{\mathbb{R}^3}\frac{C_1G\varepsilon^2}{(a+\bar{\rho})(m+\bar{k})}\mbox{d}x
\\&-\int_{\mathbb{R}^3}\frac{C_2\varepsilon^3}{m+\bar{k}}\mbox{d}x-\int_{\mathbb{R}^3}u\cdot\nabla \varepsilon\cdot \varepsilon\mbox{d}x.
\end{split}
\end{equation}
A direct computation gives that
\begin{equation}\label{temperature1}
\int_{\mathbb{R}^3}\frac{1}{(a+\bar{\rho})^2}\nabla a\cdot\nabla h\cdot h\mbox{d}x\leq C\|\nabla a\|_2\|\nabla h\|_2\|h\|_{L^{\infty}}\leq C\delta\|(\nabla a, \nabla h)\|_2^2,
\end{equation}
\begin{equation}\label{temperature2}
-\int_{\mathbb{R}^3}f'(a+\bar{\rho})(a+\bar{\rho})\mbox{div}u h\mbox{d}x\leq C\|\nabla a\|_2\|\nabla u\|_2\|h\|_3\leq C\delta \|(\nabla a, \nabla u)\|_2^2,
\end{equation}
\begin{equation}\label{temperature3}
\begin{split}
&\int_{\mathbb{R}^3}\frac{1}{a+\bar{\rho}}S_k\cdot h\mbox{d}x\\
&=\int_{\mathbb{R}^3}\frac{1}{a+\bar{\rho}}\{[\mu(\frac{\partial u^i}{\partial x_j}+\frac{\partial u^j}{\partial x_i})-\frac{2}{3}\delta_{ij}\mu\frac{\partial u^k}{\partial x_k}]\frac{\partial u^i}{\partial x_j}+\frac{\mu_t}{(a+\bar{\rho})^2}\frac{\partial p}{\partial x_j}\frac{\partial a}{\partial x_j}\}\cdot h\mbox{d}x\\
&\leq C\delta \|(\nabla a,\nabla u)\|_2^2,
\end{split}
\end{equation}
\begin{equation}\label{temperature4}
-\int_{\mathbb{R}^3}u\cdot\nabla h\cdot h\mbox{d}x\leq C\|u\|_3\|\nabla h\|_2\|h\|_6\leq C\delta \|\nabla h\|_2^2,
\end{equation}
\begin{equation}\label{kinetic1}
\int_{\mathbb{R}^3}\frac{1}{(a+\bar{\rho})^2}\nabla m\cdot\nabla a\cdot m\mbox{d}x\leq C\|\nabla m\|_2\|\nabla a\|_2\|m\|_{L^{\infty}}\leq C\delta\|(\nabla a, \nabla m)\|_2^2,
\end{equation}
\begin{equation}\label{kinetic2}
-\int_{\mathbb{R}^3}\varepsilon\cdot m\mbox{d}x\leq C\|\varepsilon\|_2\|m\|_6\|m+\bar{k}\|_6\|m+\bar{k}\|_6\leq C\delta \|\nabla m\|_2^2,
\end{equation}
\begin{equation}\label{kinetic3}
\begin{split}
&\int_{\mathbb{R}^3}\frac{1}{a+\bar{\rho}}G\cdot m\mbox{d}x\\
&=\int_{\mathbb{R}^3}\frac{1}{a+\bar{\rho}}\frac{\partial u^i}{\partial x_j}[\mu_e(\frac{\partial u^i}{\partial x_j}+\frac{\partial u^j}{\partial x_i})-\frac{2}{3}\delta_{ij}((a+\bar{\rho})(m+ \bar{k})+\mu_e\frac{\partial u^k}{\partial x_k})]\cdot m\mbox{d}x\\
&\leq C\delta \|(\nabla u,\nabla m)\|_2^2,
\end{split}
\end{equation}
\begin{equation}\label{kinetic4}
-\int_{\mathbb{R}^3}u\cdot\nabla m\cdot m\mbox{d}x\leq C\|u\|_3\|\nabla m\|_2\|m\|_6\leq C\delta \|\nabla m\|_2^2,
\end{equation}
\begin{equation}\label{varepsilon1}
\int_{\mathbb{R}^3}\frac{1}{(a+\bar{\rho})^2}\nabla \varepsilon\cdot\nabla a\cdot \varepsilon\mbox{d}x\leq C\|\nabla \varepsilon\|_2\|\nabla a\|_2\|\varepsilon\|_{L^{\infty}}\leq C\delta\|(\nabla a, \nabla \varepsilon)\|_2^2,
\end{equation}
\begin{equation}\label{varepsilon2}
\begin{split}
&\int_{\mathbb{R}^3}\frac{C_1G\varepsilon^2}{(a+\bar{\rho})(m+\bar{k})}\mbox{d}x\\
&= \int_{\mathbb{R}^3}\frac{C_1\varepsilon^2}{(a+\bar{\rho})(m+\bar{k})}\frac{\partial u^i}{\partial x_j}[\mu_e(\frac{\partial u^i}{\partial x_j}+\frac{\partial u^j}{\partial x_i})-\frac{2}{3}\delta_{ij}((a+\bar{\rho})(m+ \bar{k})\\
&+\mu_e\frac{\partial u^k}{\partial x_k})]\mbox{d}x\leq C\delta \|(\nabla u, \nabla \varepsilon)\|_2^2,
\end{split}
\end{equation}
\begin{equation}\label{varepsilon3}
-\int_{\mathbb{R}^3}\frac{C_2\varepsilon^3}{m+\bar{k}}\mbox{d}x
\leq C \|\varepsilon\|_2\|\varepsilon\|_6\|\varepsilon\|_6\|m+\bar{k}\|_6\leq C\delta \|\nabla \varepsilon\|_2^2,
\end{equation}
\begin{equation}\label{varepsilon4}
-\int_{\mathbb{R}^3}u\cdot\nabla \varepsilon\cdot \varepsilon\mbox{d}x\leq C\|u\|_3\|\nabla \varepsilon\|_2\|\varepsilon\|_6\leq C\delta \|\nabla \varepsilon\|_2^2.
\end{equation}
The estimates (\ref{temperature1})-(\ref{varepsilon4}) together with (\ref{smalltemperature1})-(\ref{smallvarepsilon1}) imply
\begin{equation}\label{smalltemperature11}
\frac{1}{2}\frac{d}{dt}\|h\|_{2}^{2}+\frac{1}{a+\bar{\rho}}\|\nabla h\|_2^2\leq C\delta \|(\nabla a, \nabla u)\|_2^2,
\end{equation}
\begin{equation}\label{smallkinetic11}
\frac{1}{2}\frac{d}{dt}\|m\|_{2}^{2}+\frac{1}{a+\bar{\rho}}\|\nabla m\|_2^2\leq C\delta  \|(\nabla a, \nabla u)\|_2^2,
\end{equation}
\begin{equation}\label{smallvarepsilon11}
\frac{1}{2}\frac{d}{dt}\|\varepsilon\|_{2}^{2}+\frac{1}{a+\bar{\rho}}\|\nabla \varepsilon\|_2^2\leq C\delta  \|(\nabla a, \nabla u)\|_2^2.
\end{equation}

Step 2:  $L^2$-norms of $\nabla^3 u,\ \nabla^3 h,\ \nabla^3 m,\ \nabla^3 \varepsilon$

Applying the differential operator $\partial_{lmn}$ to the momentum equation of (\ref{MHD1}), then multiplying it by $\partial_{lmn}u$,
and integrating over $\mathbb{R}^3$, one gets
\begin{equation}\label{smallvelocity111j}
\begin{split}
&\frac{1}{2}\frac{d}{dt}\|\partial_{lmn}u\|_{2}^{2}\\
&=\int_{\mathbb{R}^3}\partial_{lmn}(\frac{1}{a+\bar{\rho}}\Delta u)\partial_{lmn}u\mbox{d}x
+\int_{\mathbb{R}^3}\partial_{lmn}(\frac{1}{a+\bar{\rho}}\nabla \mbox{div}u)\partial_{lmn}u\mbox{d}x\\
&-\int_{\mathbb{R}^3}\partial_{lmn}(u\cdot\nabla u)\cdot\partial_{lmn}u\mbox{d}x+\int_{\mathbb{R}^3}\partial_{lmn}
[f(a+\bar{\rho})-f(\bar{\rho})]\cdot\partial_{lmn}\mbox{div}u\mbox{d}x
\\
&-\frac{2}{3}\int_{\mathbb{R}^3}\partial_{lmn}[\frac{1}{a+\bar{\rho}}
\nabla((a+\bar{\rho})(m+\bar{k}))]\cdot\partial_{lmn}u\mbox{d}x.
\end{split}
\end{equation}
The first term on the right-hand side of (\ref{smallvelocity111j}) can be estimated as
\begin{equation}\label{velocity5}
\begin{split}
&\int_{\mathbb{R}^3}\partial_{lmn}(\frac{1}{a+\bar{\rho}}\Delta u)\partial_{lmn}u\mbox{d}x\\
&=-\int_{\mathbb{R}^3}\frac{1}{a+\bar{\rho}}|\partial_{lmn}\nabla u|^2\mbox{d}x
+\int_{\mathbb{R}^3}\frac{1}{(a+\bar{\rho})^2}\nabla\partial_{lmn}u:(\nabla a \otimes \partial_{lmn}u)\mbox{d}x\\
&+\int_{\mathbb{R}^3}[-\frac{6}{(a+\bar{\rho})^4}\partial_la\partial_ma\partial_na\Delta u
+\frac{2}{(a+\bar{\rho})^3}(\partial_{lm}a\partial_na\Delta u+\partial_{ln}a\partial_ma\Delta u\\
&+\partial_{mn}a\partial_la\Delta u+\partial_la\partial_ma\partial_n\Delta u
+\partial_la\partial_na\partial_m\Delta u+\partial_na\partial_ma\partial_l\Delta u)
\\
&-\frac{1}{(a+\bar{\rho})^2}(\partial_{lmn}a\Delta u+\partial_{lm}a\partial_n\Delta u
+\partial_{ln}a\partial_m\Delta u+\partial_{mn}a\partial_l\Delta u
\\
&+\partial_{l}a\partial_{mn}\Delta u+\partial_{m}a\partial_{ln}\Delta u+\partial_{n}a\partial_{lm}\Delta u)]\cdot \partial_{lmn}u\mbox{d}x\\
&\leq C \delta \|(\nabla^2u, \nabla^3u,\nabla^4u,\nabla^2a,\nabla^3a)\|_2^2-\int_{\mathbb{R}^3}\frac{1}{a+\bar{\rho}}|\partial_{lmn}\nabla u|^2\mbox{d}x.
\end{split}
\end{equation}
Similarly, we can estimate the second and third term on the right-hand side of (\ref{smallvelocity111j}) as
\begin{equation}\label{velocity6}
\begin{split}
&\int_{\mathbb{R}^3}\partial_{lmn}(\frac{1}{a+\bar{\rho}}\nabla \mbox{div}u)\partial_{lmn}u\mbox{d}x\\
&\leq C\delta \|(\nabla^2u, \nabla^3u,\nabla^4u,\nabla^2a,\nabla^3a)\|_2^2-\int_{\mathbb{R}^3}\frac{1}{a+\bar{\rho}}|\partial_{lmn}\mbox{div} u|^2\mbox{d}x,
\end{split}
\end{equation}
\begin{equation}\label{velocity7}
-\int_{\mathbb{R}^3}\partial_{lmn}(u\cdot\nabla u)\cdot\partial_{lmn}u\mbox{d}x\leq C\delta \|\nabla^3u\|_2^2.
\end{equation}
Now, let's estimate the fourth term on the right-hand side of (\ref{smallvelocity111j}) as
\begin{equation}\label{velocity8}
\begin{split}
&\int_{\mathbb{R}^3}\partial_{lmn}[f(a+\bar{\rho})-f(\bar{\rho})]\cdot\partial_{lmn}\mbox{div}u\mbox{d}x\\
&=-\int_{\mathbb{R}^3}\partial_{lmn}[f(a+\bar{\rho})-f(\bar{\rho})]\partial_{lmn}(\frac{a_t+u\cdot\nabla a}{a+\bar{\rho}})\mbox{d}x\\
&-\int_{\mathbb{R}^3}[f'''(\rho)\partial_la\partial_ma\partial_na
+f''(\rho)\partial_{lm}a\partial_na+f''(\rho)\partial_{ln}a\partial_ma\\
&+
f''(\rho)\partial_{mn}a\partial_la
+f'(\rho)\partial_{lmn}a]
\times[-\frac{6}{(a+\bar{\rho})^4}\partial_la\partial_ma\partial_na(a_t+u\cdot\nabla a)\\
&+\frac{2}{(a+\bar{\rho})^3}(\partial_{lm}a\partial_naa_t
+\partial_{ln}a\partial_maa_t
+\partial_{mn}a\partial_laa_t
+\partial_{l}a\partial_ma\partial_na_t\\
&+\partial_{l}a\partial_na\partial_ma_t
+\partial_{n}a\partial_ma\partial_la_t
+u\cdot\nabla a\partial_{lm}a\partial_na
+u\cdot\nabla a\partial_{ln}a\partial_ma\\
&+u\cdot\nabla a\partial_{mn}a\partial_la
+\partial_lu\cdot\nabla a\partial_{m}a\partial_na
+\partial_mu\cdot\nabla a\partial_{l}a\partial_na
+\partial_nu\cdot\nabla a\partial_{l}a\partial_ma\\&+u\cdot\partial_l\nabla a\partial_{m}a\partial_na
+u\cdot\partial_m\nabla a\partial_{l}a\partial_na
+u\cdot\partial_n\nabla a\partial_{l}a\partial_ma)
\\
&-\frac{1}{(a+\bar{\rho})^2}(a_t\partial_{lmn}a
+\partial_la_t\partial_{mn}a+\partial_ma_t\partial_{ln}a
+\partial_na_t\partial_{lm}a
+\partial_{lm}a_t\partial_{n}a
\\&+\partial_{ln}a_t\partial_{m}a
+\partial_{mn}a_t\partial_{l}a
+u\cdot\nabla a\partial_{lmn}a+\partial_lu\cdot\nabla a\partial_{mn}a
\\&+\partial_mu\cdot\nabla a\partial_{ln}a
+\partial_nu\cdot\nabla a\partial_{lm}a
+u\cdot\nabla\partial_la\partial_{mn}a
+u\cdot\nabla\partial_ma\partial_{ln}a
\\&+u\cdot\nabla\partial_na\partial_{lm}a
+\partial_{lm}u\cdot\nabla a \partial_na
+\partial_{ln}u\cdot\nabla a \partial_ma
+\partial_{mn}u\cdot\nabla a \partial_la
\\
&+u\cdot \nabla\partial_{lm}a\partial_na
+u\cdot \nabla\partial_{ln}a\partial_ma
+u\cdot \nabla\partial_{mn}a\partial_la
+\partial_lu\cdot\nabla\partial_ma\partial_na
\\
&+\partial_lu\cdot\nabla\partial_na\partial_ma
+\partial_mu\cdot\nabla\partial_la\partial_na
+\partial_mu\cdot\nabla\partial_na\partial_la
+\partial_nu\cdot\nabla\partial_la\partial_ma
\\&+\partial_nu\cdot\nabla\partial_ma\partial_la)
+\frac{1}{a+\bar{\rho}}(\partial_{lmn}a_t
+\partial_{lmn}u\cdot\nabla a
+\partial_{lm}u\cdot\nabla\partial_na
\\&+\partial_{ln}u\cdot\nabla\partial_ma
+\partial_{mn}u\cdot\nabla\partial_la
+\partial_lu\cdot\nabla\partial_{mn}a
+\partial_mu\cdot\nabla\partial_{ln}a
\\&+\partial_nu\cdot\nabla\partial_{lm}a+u\cdot\nabla\partial_{lmn}a)]\mbox{d}x\\
&\leq C\delta \|(\nabla a, \nabla^2a,\nabla a_t, \nabla^3a,\nabla^2a_t,\nabla^2u,\nabla^3u)\|_2^2-\int_{\mathbb{R}^3}[f'''(\rho)\partial_la\partial_ma\partial_na
\\&+f''(\rho)\partial_{lm}a\partial_na
+f''(\rho)\partial_{ln}a\partial_ma+
f''(\rho)\partial_{mn}a\partial_la+f'(\rho)\partial_{lmn}a]\\
&\times[\frac{1}{a+\bar{\rho}}(\partial_{lmn}a_t+u\cdot\nabla\partial_{lmn}a)]\mbox{d}x.
\end{split}
\end{equation}
Since
\begin{equation*}
\begin{split}
-\int_{\mathbb{R}^3}\frac{1}{a+\bar{\rho}}f'''(\rho)\partial_la\partial_ma\partial_na\partial_{lmn}a_t\mbox{d}x
&=\int_{\mathbb{R}^3}\partial_l(\frac{1}{a+\bar{\rho}}f'''(\rho)\partial_la\partial_ma\partial_na)\partial_{mn}a_t\mbox{d}x\\
&\leq C\delta \|(\nabla a, \nabla^2a,\nabla^2a_t)\|_2^2,
\end{split}
\end{equation*}
\begin{equation*}
\begin{split}
-\int_{\mathbb{R}^3}\frac{1}{a+\bar{\rho}}f''(\rho)\partial_{lm}a\partial_na\partial_{lmn}a_t\mbox{d}x
&=\int_{\mathbb{R}^3}\partial_l(\frac{1}{a+\bar{\rho}}f''(\rho)\partial_{lm}a\partial_na)\partial_{mn}a_t\mbox{d}x\\
&\leq C\delta \|(\nabla^2 a, \nabla^2a_t,\nabla^3a)\|_2^2,
\end{split}
\end{equation*}
\begin{equation*}
\begin{split}
-\int_{\mathbb{R}^3}\frac{1}{a+\bar{\rho}}f'''(\rho)\partial_la\partial_ma\partial_nau\cdot\nabla\partial_{lmn}a\mbox{d}x
&=\int_{\mathbb{R}^3}\partial_j(\frac{1}{a+\bar{\rho}}f'''(\rho)\partial_la\partial_ma\partial_nau^j)\partial_{lmn}a\mbox{d}x
\\&\leq C\delta \|(\nabla^2a,\nabla^3a)\|_2^2,
\end{split}
\end{equation*}
and
\begin{equation*}
\begin{split}
-\int_{\mathbb{R}^3}\frac{1}{a+\bar{\rho}}f''(\rho)\partial_{lm}a\partial_nau\cdot\nabla\partial_{lmn}a\mbox{d}x
&=\int_{\mathbb{R}^3}\partial_j(\frac{1}{a+\bar{\rho}}f''(\rho)\partial_{lm}a\partial_nau^j)\partial_{lmn}a\mbox{d}x
\\&\leq C\delta \|\nabla^3a\|_2^2,
\end{split}
\end{equation*}
 we obtain
\begin{equation}\label{velocity9}
\begin{split}
&-\int_{\mathbb{R}^3}[f'''(\rho)\partial_la\partial_ma\partial_na
+f''(\rho)\partial_{lm}a\partial_na
+f''(\rho)\partial_{ln}a\partial_ma+
f''(\rho)\partial_{mn}a\partial_la]\\&
\times[\frac{1}{a+\bar{\rho}}(\partial_{lmn}a_t+u\cdot\nabla\partial_{lmn}a)]\\
&\leq C\delta|(\nabla a, \nabla^2a, \nabla^2a_t,\nabla^3a)\|_2^2.
\end{split}
\end{equation}
Finally,
\begin{equation*}
\begin{split}
&-\int_{\mathbb{R}^3}\frac{f'(\rho)}{a+\bar{\rho}}\partial_{lmn}a(\partial_{lmn}a_t+u\cdot\nabla\partial_{lmn}a)\mbox{d}x\\
&=-\frac{1}{2}\int_{\mathbb{R}^3}\frac{f'(\rho)}{a+\bar{\rho}}(\partial_{lmn}a)^2_t\mbox{d}x
+\frac{1}{2}\int_{\mathbb{R}^3}\frac{f'(\rho)}{a+\bar{\rho}}(\partial_{lmn}a)^2\mbox{div}u+
\frac{f''(\rho)}{a+\bar{\rho}}(\partial_{lmn}a)^2u\cdot\nabla a\\
&-\frac{f'(\rho)}{(a+\bar{\rho})^2}(\partial_{lmn}a)^2u\cdot \nabla a\mbox{d}x\\
&\leq -\frac{1}{2}\frac{\mbox{d}}{\mbox{d}t}\int_{\mathbb{R}^3}\frac{f'(\rho)}{a+\bar{\rho}}(\partial_{lmn}a)^2\mbox{d}x
+\frac{1}{2}\int_{\mathbb{R}^3}[\frac{f'(\rho)}{a+\bar{\rho}}]_t(\partial_{lmn}a)^2\mbox{d}x+C\delta \|\nabla^3 a\|_2^2\\
&\leq -\frac{1}{2}\frac{\mbox{d}}{\mbox{d}t}\int_{\mathbb{R}^3}\frac{f'(\rho)}{a+\bar{\rho}}(\partial_{lmn}a)^2\mbox{d}x
+C\delta \|\nabla^3 a\|_2^2,
\end{split}
\end{equation*}
together with (\ref{velocity9}), thus (\ref{velocity8}) can be replaced by
\begin{equation}\label{velocity10}
\begin{split}
&\int_{\mathbb{R}^3}\partial_{lmn}[f(a+\bar{\rho})-f(\bar{\rho})]\cdot\partial_{lmn}\mbox{div}u\mbox{d}x\\
&\leq C\delta\|(\nabla a, \nabla a_t, \nabla^2a, \nabla^2a_t, \nabla^3a,\nabla^2u,\nabla^3u)\|_2^2\\&
-\frac{1}{2}\frac{\mbox{d}}{\mbox{d}t}\int_{\mathbb{R}^3}\frac{f'(\rho)}{a+\bar{\rho}}(\partial_{lmn}a)^2\mbox{d}x.
\end{split}
\end{equation}
The last term on the right-hand side of (\ref{smallvelocity111j}) can be estimated as
\begin{equation*}
\begin{split}
&-\frac{2}{3}\int_{\mathbb{R}^3}\partial_{lmn}[\frac{1}{a+\bar{\rho}}\nabla((a+\bar{\rho})(m+\bar{k}))]\cdot\partial_{lmn}u\mbox{d}x
\\&=\frac{2}{3}\int_{\mathbb{R}^3}\partial_{mn}[\frac{m+\bar{k}}{a+\bar{\rho}}\nabla a+\nabla m]\cdot\partial_{llmn}u\mbox{d}x\\
&=\frac{2}{3}\int_{\mathbb{R}^3}[\frac{2(m+\bar{k})}{(a+\bar{\rho})^3}\partial_ma\partial_na\nabla a-\frac{1}{(a+\bar{\rho})^2}((m+\bar{k})\partial_{mn}a\nabla a
+\partial_ma\partial_nm\nabla a\\&+\partial_na\partial_mm\nabla a
+(m+\bar{k})\partial_ma\partial_n\nabla a+(m+\bar{k})\partial_na\partial_m\nabla a)
+\frac{1}{a+\bar{\rho}}(\partial_{mn}m\nabla a\\&+\partial_mm\nabla\partial_na
+\partial_nm\nabla \partial_ma
+(m+\bar{k})\nabla \partial_{mn}a)+\partial_{mn}\nabla m]\cdot\partial_{llmn}u\mbox{d}x\\
&\leq C\delta \|(\nabla a, \nabla^2a, \nabla^2m, \nabla^3m, \nabla^4m,\nabla^4u)\|_2^2+C\int_{\mathbb{R}}(\nabla \partial_{mn}a)^2\mbox{d}x,
\end{split}
\end{equation*}
which together with (\ref{velocity5})-(\ref{velocity7}) and (\ref{velocity10}) gives
\begin{equation}\label{smallvelocity111jj}
\begin{split}
&\frac{1}{2}\frac{d}{dt}[\|\partial_{lmn}u\|_{2}^{2}+\int_{\mathbb{R}^3}\frac{f'(\rho)}{a+\bar{\rho}}(\partial_{lmn}a)^2\mbox{d}x
]+\int_{\mathbb{R}^3}\frac{1}{a+\bar{\rho}}|\partial_{lmn}\nabla u|^2\mbox{d}x
\\&+\int_{\mathbb{R}^3}\frac{1}{a+\bar{\rho}}|\partial_{lmn}\mbox{div} u|^2\mbox{d}x\\
&\leq  C\delta \|(\nabla a, \nabla a_t, \nabla^2a, \nabla^2a_t, \nabla^3 a,\nabla^2u,\nabla^3u, \nabla^2m, \nabla^3m, \nabla^4m)\|_2^2
\\
&+C\int_{\mathbb{R}}(\nabla \partial_{mn}a)^2\mbox{d}x.
\end{split}
\end{equation}
Applying $\partial_{lmn}$ to the energy equation, $k$-equation and $\varepsilon$-equation of (\ref{MHD1}), multiplying the resulting equations by
$\partial_{lmn}h$, $\partial_{lmn}m$ and $\partial_{lmn}\varepsilon$ respectively, then integrating them over the whole space $\mathbb{R}^3$, one can
prove that
\begin{equation}\label{smalltemperature1j}
\begin{split}
&\frac{1}{2}\frac{d}{dt}\|\partial_{lmn}h\|_{2}^{2}\\
&=\int_{\mathbb{R}^3}\partial_{lmn}(\frac{1}{a+\bar{\rho}}\Delta h)\cdot \partial_{lmn}h\mbox{d}x+\int_{\mathbb{R}^3}\partial_{lmn}(\frac{S_k}{a+\bar{\rho}})\cdot \partial_{lmn}h\mbox{d}x
\\&-\int_{\mathbb{R}^3}\partial_{lmn}(f'(a+\bar{\rho})(a+\bar{\rho})\mbox{div}u)\cdot\partial_{lmn} h\mbox{d}x-\int_{\mathbb{R}^3}\partial_{lmn}(u\cdot\nabla h)\cdot \partial_{lmn}h\mbox{d}x,
\end{split}
\end{equation}
\begin{equation}\label{smallkinetic1j}
\begin{split}
&\frac{1}{2}\frac{d}{dt}\|\partial_{lmn}m\|_{2}^{2}\\
&=\int_{\mathbb{R}^3}\partial_{lmn}(\frac{1}{a+\bar{\rho}}\Delta m)\cdot \partial_{lmn}m\mbox{d}x
-\int_{\mathbb{R}^3}\partial_{lmn}\varepsilon\cdot \partial_{lmn}m\mbox{d}x
\\&+\int_{\mathbb{R}^3}\partial_{lmn}(\frac{G}{a+\bar{\rho}})\cdot \partial_{lmn}m\mbox{d}x-\int_{\mathbb{R}^3}\partial_{lmn}(u\cdot\nabla m)\cdot \partial_{lmn}m\mbox{d}x,
\end{split}
\end{equation}
\begin{equation}\label{smallvarepsilon1j}
\begin{split}
&\frac{1}{2}\frac{d}{dt}\|\partial_{lmn}\varepsilon\|_{2}^{2}\\
&=\int_{\mathbb{R}^3}\partial_{lmn}(\frac{1}{(a+\bar{\rho})^2}\Delta \varepsilon)\cdot \partial_{lmn}\varepsilon\mbox{d}x
+\int_{\mathbb{R}^3}\partial_{lmn}(\frac{C_1G\varepsilon}{(a+\bar{\rho})(m+\bar{k})})\cdot\partial_{lmn}\varepsilon\mbox{d}x
\\&-\int_{\mathbb{R}^3}\partial_{lmn}(\frac{C_2\varepsilon^2}{m+\bar{k}})\cdot\partial_{lmn}\varepsilon\mbox{d}x
-\int_{\mathbb{R}^3}\partial_{lmn}(u\cdot\nabla \varepsilon)\cdot \partial_{lmn}\varepsilon\mbox{d}x.
\end{split}
\end{equation}
As same as the estimate (\ref{velocity5}), we can deduce
 \begin{equation}\label{enthalpyfirst}
 \begin{split}
&\int_{\mathbb{R}^3}\partial_{lmn}(\frac{1}{a+\bar{\rho}}\Delta h)\partial_{lmn}h\mbox{d}x\\
&\leq C \delta \|(\nabla^2h, \nabla^3h,\nabla^4h,\nabla^2a,\nabla^3a)\|_2^2-\int_{\mathbb{R}^3}\frac{1}{a+\bar{\rho}}|\partial_{lmn}\nabla h|^2\mbox{d}x,
\end{split}
\end{equation}
 \begin{equation}\label{kineticfirst}
 \begin{split}
&\int_{\mathbb{R}^3}\partial_{lmn}(\frac{1}{a+\bar{\rho}}\Delta m)\partial_{lmn}m\mbox{d}x\\
&\leq C \delta \|(\nabla^2m, \nabla^3m,\nabla^4m,\nabla^2a,\nabla^3a)\|_2^2-\int_{\mathbb{R}^3}\frac{1}{a+\bar{\rho}}|\partial_{lmn}\nabla m|^2\mbox{d}x,
\end{split}
\end{equation}
 \begin{equation}\label{varepsilonfirst}
 \begin{split}
&\int_{\mathbb{R}^3}\partial_{lmn}(\frac{1}{a+\bar{\rho}}\Delta \varepsilon)\partial_{lmn}\varepsilon\mbox{d}x\\
&\leq C \delta \|(\nabla^2\varepsilon, \nabla^3\varepsilon,\nabla^4\varepsilon,\nabla^2a,\nabla^3a)\|_2^2-\int_{\mathbb{R}^3}\frac{1}{a+\bar{\rho}}|\partial_{lmn}\nabla \varepsilon|^2\mbox{d}x.
\end{split}
\end{equation}
Again, from the H\"{o}lder inequality, (\ref{smallcondition3})-(\ref{density4}),
 $\dot{W}^{1,2}(\mathbb{R}^3)\hookrightarrow L^6(\mathbb{R}^3)$ and $H^2(\mathbb{R}^3)\hookrightarrow L^{\infty}(\mathbb{R}^3)$, we get
 \begin{equation*}
\begin{split}
&\int_{\mathbb{R}^3}\partial_{lmn}(\frac{S_k}{a+\bar{\rho}})\cdot \partial_{lmn}h\mbox{d}x
\\&=\int_{\mathbb{R}^3}\partial_{lmn}\big{\{}\frac{1}{a+\bar{\rho}}\big[\big(\mu(\frac{\partial u^i}{\partial x_j}
+\frac{\partial u^j}{\partial x_i})-\frac{2}{3}\delta_{ij}\mu\frac{\partial u^k}{\partial x_k}\big)\frac{\partial u^i}{\partial x_j}
+\frac{\mu_t}{(a+\bar{\rho})^2}\frac{\partial p}{\partial x_j}\frac{\partial a}{\partial x_j}\big]\big{\}}\cdot \partial_{lmn}h\mbox{d}x\\
&\leq C\delta \|(\nabla^3a,\nabla^3u,\nabla^4u,\nabla^3h,\nabla^4h)\|_2^2,
\end{split}
\end{equation*}
\begin{equation*}
\begin{split}
&\int_{\mathbb{R}^3}\partial_{lmn}(\frac{G}{a+\bar{\rho}})\cdot \partial_{lmn}m\mbox{d}x\\
&=\int_{\mathbb{R}^3}\partial_{lmn}\big{\{}\frac{1}{a+\bar{\rho}}\frac{\partial u^i}{\partial x_j}\big[\mu_e(\frac{\partial u^i}{\partial x_j}
+\frac{\partial u^j}{\partial x_i})-\frac{2}{3}\delta_{ij}\big((a+\bar{\rho})(m+ \bar{k})
+\mu_e\frac{\partial u^k}{\partial x_k}\big)\big]\big{\}}\cdot \partial_{lmn}m\mbox{d}x\\
& \leq C\delta \|(\nabla^2a,\nabla^3a,\nabla^3m,\nabla^4m,\nabla^3u,\nabla^4u)\|_2^2,
\end{split}
\end{equation*}
\begin{equation*}
\begin{split}
&\int_{\mathbb{R}^3}\partial_{lmn}(\frac{C_1G\varepsilon}{(a+\bar{\rho})(m+\bar{k})})\cdot\partial_{lmn}\varepsilon\mbox{d}x\\
&=\int_{\mathbb{R}^3}\partial_{lmn}\big{\{}\frac{C_1\varepsilon}{(a+\bar{\rho})(m+\bar{k})}\frac{\partial u^i}{\partial x_j}\big[\mu_e(\frac{\partial u^i}{\partial x_j}
+\frac{\partial u^j}{\partial x_i})-\frac{2}{3}\delta_{ij}\big((a+\bar{\rho})(m+ \bar{k})
\\&+\mu_e\frac{\partial u^k}{\partial x_k}\big)\big]\big{\}}\cdot\partial_{lmn}\varepsilon\mbox{d}x\\&
\leq C\delta \|(\nabla^3a,\nabla^3u,\nabla^4u,\nabla^3m,\nabla^3\varepsilon)\|_2^2,
\end{split}
\end{equation*}
 \begin{equation*}
 -\int_{\mathbb{R}^3}\partial_{lmn}(f'(a+\bar{\rho})(a+\bar{\rho})\mbox{div}u)\cdot\partial_{lmn} h\mbox{d}x
 \leq C\delta \|(\nabla^3a,\nabla^3u,\nabla^4u,\nabla^3h,\nabla^4h)\|_2^2,
 \end{equation*}
\begin{equation*}
-\int_{\mathbb{R}^3}\partial_{lmn}(u\cdot\nabla h)\cdot \partial_{lmn}h\mbox{d}x\leq C\delta \|(\nabla^3u,\nabla^3h,\nabla^4h)\|_2^2,
\end{equation*}
\begin{equation*}
-\int_{\mathbb{R}^3}\partial_{lmn}\varepsilon\cdot \partial_{lmn}m\mbox{d}x\leq C\delta \|(\nabla^3\varepsilon,\nabla^3m,\nabla^4m)\|_2^2,
\end{equation*}
\begin{equation*}
-\int_{\mathbb{R}^3}\partial_{lmn}(u\cdot\nabla m)\cdot \partial_{lmn}m\mbox{d}x\leq
C\delta \|(\nabla^3m,\nabla^4m,\nabla^3u)\|_2^2,
\end{equation*}
\begin{equation*}
-\int_{\mathbb{R}^3}\partial_{lmn}(\frac{C_2\varepsilon^2}{m+\bar{k}})\cdot\partial_{lmn}\varepsilon\mbox{d}x
\leq C\delta\|(\nabla^3m,\nabla^3\varepsilon)\|_2^2,
\end{equation*}
\begin{equation*}
-\int_{\mathbb{R}^3}\partial_{lmn}(u\cdot\nabla \varepsilon)\cdot \partial_{lmn}\varepsilon\mbox{d}x
\leq C \|(\nabla^3\varepsilon,\nabla^4\varepsilon,\nabla^3u)\|_2^2.
\end{equation*}
Incorporating the above estimates and (\ref{smalltemperature1j})-(\ref{smallvarepsilon1j}) yields that
\begin{equation}\label{smalltemperature1jj}
\begin{split}
\frac{1}{2}\frac{d}{dt}\|\partial_{lmn}h\|_{2}^{2}&+\int_{\mathbb{R}^3}\frac{1}{a+\bar{\rho}}|\partial_{lmn}\nabla h|^2\mbox{d}x
\\&\leq C\delta\|(\nabla^2a,\nabla^3a, \nabla^2h,\nabla^3h,\nabla^3u,\nabla^4u)\|_2^2,
\end{split}
\end{equation}
\begin{equation}\label{smallkinetic1jj}
\begin{split}
\frac{1}{2}\frac{d}{dt}\|\partial_{lmn}m\|_{2}^{2}&+\int_{\mathbb{R}^3}\frac{1}{a+\bar{\rho}}|\partial_{lmn}\nabla m|^2\mbox{d}x
\\&\leq C\delta\|(\nabla^2a,\nabla^3a,\nabla^2m,\nabla^3m,\nabla^3\varepsilon,\nabla^3u,\nabla^4u)\|_2^2,
\end{split}
\end{equation}
\begin{equation}\label{smallvarepsilon1jj}
\begin{split}
\frac{1}{2}\frac{d}{dt}\|\partial_{lmn}\varepsilon\|_{2}^{2}&+\int_{\mathbb{R}^3}\frac{1}{a+\bar{\rho}}|\partial_{lmn}\nabla \varepsilon|^2\mbox{d}x\\
&\leq C\delta\|(\nabla^2a,\nabla^3a,\nabla^3m,\nabla^2\varepsilon,\nabla^3\varepsilon,\nabla^3u,\nabla^4u)\|_2^2.
\end{split}
\end{equation}

Step 3:  $L^2$-norms of $\nabla a,\ \nabla^3a$

We first estimate for $\nabla a$. For this purpose, we calculate as
\begin{equation}\label{smalldensity}
\begin{split}
&\int_{\mathbb{R}^3}[\frac{1}{2}|\nabla a|^2+\frac{(a+\bar{\rho})^2}{2}\nabla a\cdot u]_t
\\&=\int_{\mathbb{R}^3}\nabla a\cdot\nabla a_t+(a+\bar{\rho})a_t\nabla a\cdot u
+\frac{(a+\bar{\rho})^2}{2}\nabla a_t\cdot u+\frac{(a+\bar{\rho})^2}{2}\nabla a\cdot u_t\mbox{d}x.
\end{split}
\end{equation}
The first term of the right-hand side of (\ref{smalldensity}) can be estimated as follows:
\begin{equation}\label{smalldensity1}
\begin{split}
\int_{\mathbb{R}^3}\nabla a\cdot\nabla a_t\mbox{d}x&=-\int_{\mathbb{R}^3}\nabla a\cdot\nabla\mbox{div}[(a+\bar{\rho})u]\mbox{d}x\\
&=-\int_{\mathbb{R}^3}\nabla a\cdot(\nabla^2a\cdot u)+\nabla a\cdot(\nabla a\cdot\nabla u)\\&+(a+\bar{\rho})\nabla a\cdot \nabla \mbox{div}u
+|\nabla a|^2\mbox{div}u\mbox{d}x
\\&\leq C\delta \|\nabla a\|_2^2-\int_{\mathbb{R}^3}(a+\bar{\rho})\nabla a\cdot\nabla\mbox{div}u\mbox{d}x.
\end{split}
\end{equation}
Also, we estimate the rest three terms as
\begin{equation}\label{smalldensity2}
\int_{\mathbb{R}^3}(a+\bar{\rho})a_t\nabla a\cdot u\mbox{d}x\leq C\delta \|(a_t,\nabla a)\|_2^2\leq C\delta\|(\nabla a,\nabla u)\|_2^2,
\end{equation}
\begin{equation}\label{smalldensity3}
\int_{\mathbb{R}^3}\frac{(a+\bar{\rho})^2}{2}\nabla a_t\cdot u\mbox{d}x\leq C\delta\|(\nabla a,\nabla u)\|_2^2,
\end{equation}
\begin{equation}\label{smalldensity4}
\begin{split}
&\int_{\mathbb{R}^3}\frac{(a+\bar{\rho})^2}{2}\nabla a\cdot u_t\mbox{d}x\\
&=\int_{\mathbb{R}^3}\frac{(a+\bar{\rho})^2}{2}\nabla a\cdot\big(-u\cdot\nabla u
+\frac{1}{a+\bar{\rho}}\Delta u+\frac{1}{a+\bar{\rho}}\nabla\mbox{div}u\\&-\nabla[f(a+\bar{\rho})-f(\bar{\rho})]
-\frac{2}{3(a+\bar{\rho})}\nabla ((a+\bar{\rho})(m+\bar{k}))\big)\mbox{d}x\\&
\leq C\delta\|(\nabla a,\nabla u)\|_2^2-\int_{\mathbb{R}^3}\frac{(a+\bar{\rho})(m+\bar{k})}{3}|\nabla a|^2\mbox{d}x\\&-\int_{\mathbb{R}^3}\frac{(a+\bar{\rho})^2}{2}f'(\rho)|\nabla a|^2\mbox{d}x
+\int_{\mathbb{R}^3}\frac{a+\bar{\rho}}{2}\nabla a\cdot(\Delta u+\nabla\mbox{div}u)\mbox{d}x.
\end{split}
\end{equation}
Since
\begin{equation*}
\begin{split}
&\int_{\mathbb{R}^3}\frac{a+\bar{\rho}}{2}\nabla a\cdot(\Delta u-\nabla\mbox{div}u)\mbox{d}x
=\int_{\mathbb{R}^3}\frac{a+\bar{\rho}}{2}\partial_ia\cdot(\partial_{jj} u^i-\partial_i\partial_ju^j)\mbox{d}x\\
&=\int_{\mathbb{R}^3}-\frac{1}{2}\partial_ia\partial_ja\partial_ju^i-\frac{a+\bar{\rho}}{2}\partial_{ij} a\partial_ju^i+\frac{1}{2}\partial_ia\partial_ja\partial_iu^j
+\frac{a+\bar{\rho}}{2}\partial_{ij} a\partial_iu^j\mbox{d}x=0,
\end{split}
\end{equation*}
together with (\ref{smalldensity})-(\ref{smalldensity4}) and (\ref{density3})-(\ref{density4}), we show that
\begin{equation}\label{smalldensity5}
\int_{\mathbb{R}^3}[\frac{1}{2}|\nabla a|^2+\frac{(a+\bar{\rho})^2}{2}\nabla a\cdot u]_t+C\|\nabla a\|_2^2
\leq C\delta \|\nabla u\|_2^2.
\end{equation}

Now, we turn to estimate for $\nabla^3 a$. Almost parallel to the inequality (\ref{smalldensity}),  we get
\begin{equation}\label{smalldensity6}
\begin{split}
&\int_{\mathbb{R}^3}[\frac{1}{2}|\partial_{lmn}a|^2+\frac{(a+\bar{\rho})^2}{2}\partial_{lmn} a \partial_{lm}u^n]_t\mbox{d}x\\
&=\int_{\mathbb{R}^3}-\partial_{lmn}a\partial_{lmn}\mbox{div}[(a+\bar{\rho})u]+(a+\bar{\rho})a_t\partial_{lmn}a\partial_{lm}u^n
\\&-\frac{(a+\bar{\rho})^2}{2}\partial_{lmn}\mbox{div}[(a+\bar{\rho})u]\partial_{lm}u^n
+\frac{(a+\bar{\rho})^2}{2}\partial_{lmn}a\partial_{lm}u^n_t\mbox{d}x.
\end{split}
\end{equation}
We estimate the right-hand side of (\ref{smalldensity6}) as follows:
\begin{equation}\label{smalldensity6j}
\begin{split}
&\int_{\mathbb{R}^3}-\partial_{lmn}a\partial_{lmn}\mbox{div}[(a+\bar{\rho})u]\mbox{d}x\\
&=\int_{\mathbb{R}^3}-\partial_{lmn}a[u\cdot \nabla \partial_{lmn}a+\partial_lu\cdot\nabla\partial_{mn}a+\partial_mu\cdot\nabla\partial_{ln}a
+\partial_nu\cdot\nabla\partial_{lm}a\\&+\partial_{lm}u\cdot\nabla\partial_{n}a
+\partial_{ln}u\cdot\nabla\partial_{m}a
+\partial_{mn}u\cdot\nabla\partial_{l}a+\partial_{lmn}u\cdot\nabla a+\partial_{lmn}a\mbox{div}u
\\&+\partial_{lm}a\partial_n\mbox{div}u
+\partial_{ln}a\partial_m\mbox{div}u
+\partial_{mn}a\partial_l\mbox{div}u
+\partial_{l}a\partial_{mn}\mbox{div}u+\partial_{m}a\partial_{ln}\mbox{div}u
\\&+\partial_{n}a\partial_{lm}\mbox{div}u+(a+\bar{\rho})\partial_{lmn}\mbox{div}u]\mbox{d}x\\
&\leq C\delta \|(\nabla^3a, \nabla^3u)\|_2^2-\int_{\mathbb{R}^3}(a+\bar{\rho})\partial_{lmn}a\partial_{lmn}\mbox{div}u\mbox{d}x,
\end{split}
\end{equation}
\begin{equation}\label{smalldensity7}
\int_{\mathbb{R}^3}(a+\bar{\rho})a_t\partial_{lmn}a\partial_{lm}u^n\mbox{d}x\leq C\delta \|(\nabla^3a, \nabla^2u)\|_2^2,
\end{equation}
\begin{equation}\label{smalldensity8}
\begin{split}
&-\int_{\mathbb{R}^3}\frac{(a+\bar{\rho})^2}{2}\partial_{lmn}\mbox{div}[(a+\bar{\rho})u]\partial_{lm}u^n\mbox{d}x\\
&=
-\int_{\mathbb{R}^3}\frac{(a+\bar{\rho})^2}{2}\partial_{lm}u^n[\partial_{lmn}(u\cdot\nabla a)+\partial_{lmn}a\mbox{div}u
+\partial_{lm}a\partial_{n}\mbox{div}u\\&+\partial_{ln}a\partial_{m}\mbox{div}u
+\partial_{mn}a\partial_{l}\mbox{div}u
+\partial_{l}a\partial_{mn}\mbox{div}u+\partial_{m}a\partial_{ln}\mbox{div}u+\partial_{n}a\partial_{lm}\mbox{div}u
\\&+(a+\bar{\rho})\partial_{lmn}\mbox{div}u]\mbox{d}x\\
&\leq C\delta \|(\nabla^2a, \nabla^3a, \nabla^2u, \nabla^3u, \nabla^4u)\|_2^2,
\end{split}
\end{equation}
\begin{equation}\label{smalldensity9}
\begin{split}
&\int_{\mathbb{R}^3}\frac{(a+\bar{\rho})^2}{2}\partial_{lmn}a\partial_{lm}u^n_t\mbox{d}x\\
&=\int_{\mathbb{R}^3}\frac{(a+\bar{\rho})^2}{2}\partial_{lmn}a\partial_{lm}[-u^i\partial_iu^n+\frac{1}{a+\bar{\rho}}\partial_{ii}u^n
+\frac{1}{a+\bar{\rho}}\partial_{in}u^i\\&-\partial_n(f(a+\bar{\rho})-f(\bar{\rho}))
-\frac{2}{3(a+\bar{\rho})}\partial_n((a+\bar{\rho})(m+\bar{k}))]\mbox{d}x
\\&\leq C\delta\|(\nabla^2a, \nabla^3a, \nabla^2u, \nabla^3u, \nabla^3m,\nabla^4m)\|_2^2
\\&+\int_{\mathbb{R}^3}\frac{a+\bar{\rho}}{2}\partial_{lmn}a(\partial_{lmii}u^n+\partial_{lmni}u^i)\mbox{d}x
-\int_{\mathbb{R}^3}\frac{(a+\bar{\rho})^2}{2}f'(\rho)(\partial_{lmn}a)^2\mbox{d}x\\&
-\int_{\mathbb{R}^3}\frac{(a+\bar{\rho})(m+\bar{k})}{3}(\partial_{lmn}a)^2\mbox{d}x.
\end{split}
\end{equation}
Noting that
\begin{equation*}
\begin{split}
&-\int_{\mathbb{R}}(a+\bar{\rho})\partial_{lmn}a\partial_{lmn}\mbox{div}u\mbox{d}x
+\int_{\mathbb{R}^3}\frac{a+\bar{\rho}}{2}\partial_{lmn}a(\partial_{lmii}u^n+\partial_{lmni}u^i)\mbox{d}x\\
&=\int_{\mathbb{R}^3}\frac{a+\bar{\rho}}{2}\partial_{lmn}a(\partial_{lmii}u^n-\partial_{lmni}u^i)\mbox{d}x\\
&=-\frac{1}{2}\int_{\mathbb{R}^3}\partial_ia\partial_{lmn}a\partial_{lmi}u^n+(a+\bar{\rho})\partial_{lmni}a\partial_{lmi}u^n
-\partial_{i}a\partial_{lmn}a\partial_{lmn}u^i\\&-(a+\bar{\rho})\partial_{lmni}a\partial_{lmn}u^i\mbox{d}x
\\&=-\frac{1}{2}\int_{\mathbb{R}^3}\partial_ia\partial_{lmn}a\partial_{lmi}u^n
-\partial_{i}a\partial_{lmn}a\partial_{lmn}u^i\mbox{d}x\\&
\leq C\delta \|(\nabla^3a, \nabla^3u)\|_2^2,
\end{split}
\end{equation*}
 based on the estimates (\ref{smalldensity6})-(\ref{smalldensity9}), with the help of (\ref{density3})-(\ref{density4})
and the interpolation inequality, one obtains
\begin{equation}\label{smalldensity6jj}
\begin{split}
&\int_{\mathbb{R}^3}[\frac{1}{2}|\partial_{lmn}a|^2+\frac{(a+\bar{\rho})^2}{2}\partial_{lmn} a \partial_{lm}u^n]_t\mbox{d}x+C\|\nabla^3 a\|_2^2
\\&\leq C\delta \|(\nabla^2a, \nabla^2u,\nabla^3u,\nabla^4u,\nabla^3m,\nabla^4m)\|_2^2.
\end{split}
\end{equation}

Step 4:  Conclusion

Consequently, multiplying (\ref{smallvelocity111jj}) by a appropriate small constants $\alpha$, together with \eqref{velocity}, \eqref{smalltemperature11}-\eqref{smallvarepsilon11}, (\ref{smalltemperature1jj})-(\ref{smallvarepsilon1jj}), (\ref{smalldensity5}) and (\ref{smalldensity6jj}), we have
\begin{equation}\label{incorporation}
\begin{split}
&\frac{d}{dt}\{\|(u,h, \nabla^3h, m,\nabla^3m,\varepsilon,\nabla^3\varepsilon)\|_{2}^{2}+\int_{\mathbb{R}^3}F(a)\mbox{d}x+
\alpha[\|\nabla^3u\|_2^2\\&+\int_{\mathbb{R}^3}\frac{f'(\rho)}{a+\bar{\rho}}(\partial_{lmn}a)^2\mbox{d}x
]+\int_{\mathbb{R}^3}\big[\frac{1}{2}|\nabla a|^2+\frac{(a+\bar{\rho})^2}{2}\nabla a\cdot u
\\&+\frac{1}{2}|\partial_{lmn}a|^2+\frac{(a+\bar{\rho})^2}{2}\partial_{lmn} a \partial_{lm}u^n\big]\mbox{d}x\}\\&+C(\alpha)\|(\nabla a,\nabla^3a, \nabla u, \nabla^4u, \nabla h,\nabla^4h, \nabla m,\nabla^4m, \nabla\varepsilon,\nabla^4\varepsilon)\|_2^2\\&\leq 0,
\end{split}
\end{equation}
where we have used the fact that
\begin{equation*}
|\nabla^ia_t|\leq C\sum_{k=1}^{i+1}(|\nabla^ka|+|\nabla^ku|),\ \ i=1,2.
\end{equation*}
Integrating the inequality (\ref{incorporation}), from (\ref{smallcondition3})-(\ref{density4}), (\ref{fensityfunction})
 and the smallness of $\alpha$, we can finish the proof of Proposition \ref{small theorem}.
 \end{proof}

\section {Proof of global existence}

We will finish the proof of Theorem \ref{thm1} in this section. First, let's state the local existence. Since it can be proved in a standard
way as that in \cite{Nishita, volpert}, we omit the proof.
\begin{prop}\label{thm1local}
Under the assumption of the Theorem \ref{thm1}, then there exists a constant $T>0$ such that the system (\ref{MHD1})
 admits a unique smooth solution $(\rho, u, h, k,\varepsilon)$ which satisfies that
 there exists a constant $C_2>1$ such that for any
$t\in [0,T]$,
\begin{equation}\label{localdependent initial}
\begin{split}
&\|(\rho-\bar{\rho}, u, h, k-\bar{k},\varepsilon)\|_{H^3}^2+\int_0^t(\|\nabla \rho\|_{H^2}^2+\|(\nabla u,\nabla h, \nabla k,
\nabla \varepsilon)\|_{H^3}^2)\mbox{d}s\\
&\leq C_2
\|(\rho_0-\bar{\rho}, u_{0}, h_0, k_0-\bar{k}, \varepsilon_0)\|_{H^3}^2.
\end{split}
\end{equation}
\end{prop}
In the following, by a continued argument, combining the local existence and the a priori estimates proposition, we will prove the global existence of smooth
solutions.

First, suppose
\begin{equation}\label{smallE}
E_0=\|(\rho_0-\bar{\rho}, u_{0}, h_0, k_0-\bar{k}, \varepsilon_0)\|_{H^3}<\min(\delta/\sqrt{C_2},\delta/\sqrt{C_1C_2}),
\end{equation}
 where $\delta$ is defined in Proposition \ref{small theorem}.
Since the initial data satisfy
 $E_0<\delta/\sqrt{C_2}$,
  then by Proposition \ref{thm1local}, there exists a constant
$T^{*}>0$ such that there exists a unique solution on $[0,T_{*}]$ satisfying
\begin{equation}\label{localdependent initial1}
E_1:=\sup_{0\leq t\leq T_1}\|(\rho-\bar{\rho}, u, h, k-\bar{k}, \varepsilon)\|_{H^3}\leq\sqrt{ C_2}
E_0.
\end{equation}  Therefore, using the inequality $E_0<\delta/\sqrt{C_1C_2}$, from Proposition \ref{small theorem},
 we have
 \begin{equation}\label{smallE1}
 E_1\leq\sqrt{C_1}E_0<\delta/\sqrt{C_2}.
 \end{equation}
Notice that $T^*$ depends only on $E_0$. Starting from $T^*$, then the initial problem (\ref{MHD1}) with initial data $(\rho,u,h,k,\varepsilon)(T^{*})$
still has a unique solution on $[T^*, 2T^*]$, and from Proposition \ref{thm1local}, we get
$$\sup_{T^*\leq t\leq2T^*}\|(\rho-\bar{\rho}, u, h, k-\bar{k}, \varepsilon)\|_{H^3}\leq \sqrt{C_2}E_1\leq \sqrt{C_1C_2}E_0\leq \delta.$$
Again from Proposition \ref{small theorem}, one can deduce
$$E_2=\sup_{0\leq t\leq 2T_{*}}\|(\rho-\bar{\rho}, u, h, k-\bar{k}, \varepsilon)\|_{H^3}\leq \sqrt{C_1}E_0<\delta/\sqrt{C_2}.$$
Repeating the procedure for $0\leq t\leq NT_*$, $N=1,2,3,\cdots$, we can extend the local solution to infinity as far as the initial data are small enough
such that $E_0\leq\min(\delta/\sqrt{C_2},\delta/\sqrt{C_1C_2})$.  Thus the proof of Theorem \ref{thm1} is complete.\qed

\end{document}